\documentclass[11pt,a4paper]{article}

\usepackage{times}
\usepackage{enumerate}

\usepackage[english]{babel}

\usepackage{verbatim}

\usepackage{color}

\usepackage{a4wide}

\usepackage{amsfonts}

\usepackage{amscd}
\usepackage{tabularx}
\usepackage{arydshln}
\usepackage{verbatim}
\usepackage{color}
\usepackage{a4wide}
 
\usepackage[all]{xy}
\usepackage{mathtools}

\usepackage[all]{xy}

\newtheorem{theorem}{Theorem}[section]

\newtheorem{proposition}[theorem]{Proposition}
\newtheorem{corollary}[theorem]{Corollary}
\newtheorem{Counter-example}[theorem]{Counter-example}
\newtheorem{remark}[theorem]{Remark}

\newenvironment{proof}{\noindent{\it Proof.}}{\hfill$\blacksquare$}

\newtheorem{definition}[theorem]{Definition}
\usepackage{amsmath}

\long\def\symbolfootnote[#1]#2{\begingroup\def\thefootnote{\fnsymbol{footnote}}
\footnote[#1]{#2}\endgroup}

\usepackage{amssymb}

\begin{document}

\def\Q{\mathbb Q}
\def\R{\mathbb R}
\def\N{\mathbb N}
\def\Z{\mathbb Z}
\def\C{\mathbb C}
\def\S{\mathbb S}
\def\L{\mathbb L}
\def\H{\mathbb H}
\def\K{\mathbb K}
\def\X{\mathbb X}
\def\Y{\mathbb Y}
\def\Z{\mathbb Z}
\def\E{\mathbb E}
\def\J{\mathbb J}
\def\I{\mathbb I}
\def\T{\mathbb T}
\def\H{\mathbb H}

\title{Normal tractor conformal bundles and codimension two spacelike submanifolds in Lorentzian manifolds}

\author{Rodrigo Mor\'on$^*$ and Francisco J. Palomo}





\date{}

\maketitle

\begin{abstract}
 
For every codimension two spacelike submanifold of a Lorentz manifold and each choice of a normal lightlike vector field, we introduce a canonical way to construct a tractor conformal bundle.
We characterize when the induced connection of a such submanifold defines a tractor connection
and then, in this case, when this tractor conformal bundle with the induced connection is standard for the induced metric. Finally, the normality condition for this tractor conformal bundle endowed with the induced connection is characterized in terms of a strong relationship between the intrinsic and the extrinsic geometry of the starting spacelike submanifold.
 
\end{abstract}


\hyphenation{Lo-rent-zi-an}

\section{Introduction}

\symbolfootnote[0]{
The second author is partially supported by the Regional Government of Andalusia and ERDEF project PY20-01391 and both of them by Spanish MICINN project PID2020-118452GB-100.
\newline 
$^*$ Corresponding author.

{\it 2020 Mathematical Subjet Classification.} Primary 53C40, 53B15, 53C18 Secondary 53C50.

{\it Keywords.} Conformal metrics, Normal tractor conformal bundles,  Lorentzian geometry,  Spacelike submanifolds, Totally umbilical spacelike submanifolds.

}

 A Riemannian
conformal structure on a manifold $M$ is an equivalence class $c=[g]$ of Riemannian metrics  where two metrics $g$ and $g'$ are said to be equivalent when $g'=e^{2u}g$ for a smooth function $u$ on $M$.
A classical problem in  conformal geometry is the construction of invariants. 
The invariants of any Riemannian metric $g\in c$ transform according to intricate formulas and they do not provide invariants for the conformal structure. The problem of the construction of conformal invariants remains very difficult. Anyway,
  there are several equivalent ways of describing conformal structures which immediately lead to a construction of conformal invariants. For dimension $n\geq 3$, every conformal structure $(M,c)$ admits a unique normal conformal Cartan connection (see for instance \cite[Theor. 1.6.7]{CS09}). 
On the other hand, a tractor conformal bundle $\mathcal{T}\to M$ endowed with a tractor connection $\nabla^{\mathcal{T}}$ describes conformal structures in terms of vector bundles and linear connections, see Section 2. This point of view goes back to the works of Thomas in the 1920s, \cite{Thomas}. Moreover, Fefferman and Graham introduced a new point of view in 1985, \cite{F-G}. They though a conformal structure as a ray bundle $\mathcal{Q}\to M$ and then associated an ambient Lorentzian metric on the manifold $\mathcal{Q}\times (-\varepsilon, +\varepsilon)$ which satisfies certain normalisation conditions. In particular, the embedding of $\mathcal{Q}$ at $\rho=0$ becomes a lightlike hypersurface in $\mathcal{Q}\times (-\varepsilon, +\varepsilon)$. 

The relation between tractor connections on tractor conformal  bundles and conformal Cartan connections was clarified by  \v{C}ap and Gover in \cite{Cap1} in the more general setting of parabolic geometries. There is a one-to-one correspondence between conformal tractor connections on $\mathcal{T}\to M$ and  conformal Cartan connections. The tractor connection which corresponds with the normal conformal Cartan connection is  called the normal tractor connection.
In \cite{Cap}, the normal tractor connection is characterized in terms of  curvature properties. For the sake of completeness, we have included here this result in Theorem \ref{261221A}. 
The same authors have related the ambient construction by Fefferman and Graham with the normal Cartan connection of a conformal structure in \cite{Cap}. In order to get this aim, 
 \v{C}ap and Gover introduce an ambient metric which satisfies a weakening normalisation condition of the original by Fefferman and Graham, Remark \ref{141221B}.

The purpose of this paper is to relate tractor bundles and tractor connections with codimension two  spacelike submanifolds in a Lorentzian manifold. From this point of view, this paper is a continuation of \cite{MP}.
Let us recall the following definitions to explain our point of view here.
\begin{definition}($\cite{Cap}$)
A (Riemannian) tractor conformal bundle on a manifold $M$ with $\mathrm{dim}M =n \geq 2$ is a rank $n+2$ real vector bundle $\mathcal{T}\to M$ endowed with a bundle metric $h$ of Lorentzian signature (i.e., $(n+1,1)$) and  with a distinguished lightlike line subbundle $\mathcal{T}^{1}\subset \mathcal{T}$.
\end{definition}

\begin{definition}($\cite{Burs}, \cite{Cap}$)
A tractor connection $\nabla^{\mathcal{T}}$ on a tractor conformal bundle $\mathcal{T}\to M$ is a linear connection such that 
$\nabla^{\mathcal{T}} h=0$ and the following map $\beta$ is an isomorphism of vector bundles on $M$
\begin{equation}\label{011021}
\begin{array}{rcl}
 TM & \stackrel{\beta\ \ }\longrightarrow & \mathrm{Hom}(\mathcal{T}^{1},(\mathcal{T}^{1})^{\perp}/\mathcal{T}^{1} ) \\    \searrow &   & \ \ \swarrow  \\   & M&
\end{array}
\end{equation}
where $\beta(X_{x})(v)=\nabla^{\mathcal{T}}_{X_{x}}\sigma +\,\mathcal{T}^{1}_{x}$ for $x\in M$, $X_{x}\in T_{x}M$, $v\in \mathcal{T}^{1}_{x}$ and  $\sigma\in \Gamma(\mathcal{T}^{1})$ is any section with $\sigma(x)=v$. 
\end{definition}
Every  nonvanishing section $\sigma$ of $\mathcal{T}^{1}$ provides the vector bundle isomorphism
$$
\beta_{\sigma}\colon TM\to (\mathcal{T}^{1})^{\perp}/\mathcal{T}^{1}, \quad X \mapsto \nabla^{\mathcal{T}}_{X}\sigma+ \mathcal{T}^{1}.
$$
Any other nonvanishing section of $\mathcal{T}^{1}$ is written as $f \sigma$ for a nonvanishing smooth function $f$
on $M$ and the condition $\nabla^{\mathcal{T}} h=0$ implies $\beta_{f\cdot\sigma}= f \cdot \beta_{\sigma}$. Thus, every section $\sigma \in \Gamma(\mathcal{T}^{1})$ produces a Riemannian metric $h^{\sigma}$ on $M$ by means of the formula
\begin{equation}\label{241121B}
h^{\sigma}(X,Y)=h(\beta_{\sigma}(X), \beta_{\sigma}(Y)),
\end{equation}
for $X, Y \in \mathfrak{X}(M)$. Different choices of the section $\sigma$ induce conformally related metrics on $M$ and then, a conformal structure $c$ on $M$.
If we start with a conformal structure $(M,c)$ and the induced conformal structure on $M$ by means of $(\mathcal{T},\mathcal{T}^{1}, h, \nabla^{\mathcal{T}})$ agrees with $c$, we say that $(\mathcal{T}, \mathcal{T}^{1}, h, \nabla^{\mathcal{T}})$  is a standard tractor conformal  bundle for $(M,c)$.

There is a geometric setting where several of the above mentioned objects arise in a natural way. Namely, from every spacelike submanifold $\Psi\colon M^{n}\to (\widetilde{M}^{n+2}, \widetilde{g})$ and each lightlike normal vector field $\xi$, we can construct a {\it tractor conformal bundle} as follows. The vector bundle $\mathcal{T}$  on $M$ is the pull-back via $\Psi$ of the tangent bundle of the manifold $\widetilde{M}$ with bundle metric $\widetilde{g}$
and distinguished lightlike line subbundle $\mathcal{T}^{1}=\mathrm{Spam}\{\xi\}$.
The natural choice for a tractor connection is the induced connection $\widetilde{\nabla}$.

Of course, the induced connection $\widetilde{\nabla}$ is always a metric connection but the map $\beta$ defines an isomorphism of vector bundles if and only if the Weingarten endomorphism $A_{\xi}$ corresponding to the normal vector field $\xi$ is non-singular at every point $x\in M,$ Proposition \ref{241121D}.
Then, the following natural question is about when $(\mathcal{T},\mathcal{T}^{1}, \widetilde{g} , \widetilde{\nabla})$ is a
standard  tractor conformal bundle for the conformal class of the induced metric on $M$.
That is, when the induced metric from $\Psi$ is in the equivalence class of the conformal structure deduced from $(\mathcal{T},\mathcal{T}^{1}, \widetilde{g} , \widetilde{\nabla})$?
This happens if and only if there is a smooth function $\lambda\in C^{\infty}(M)$ such that $A^{2}_{\xi}= \lambda^{2} \cdot \mathrm{Id}$ (where $\lambda$ is a non-vanishing function on $M$), Proposition \ref{241121D}. There are  two  mutually  disjoint  possibilities in order to the condition $A^{2}_{\xi}= \lambda^{2} \cdot \mathrm{Id}$ holds. Namely, $A_{\xi}= \lambda \cdot \mathrm{Id}$, that is, $\xi$ is an umbilical direction or $M$ is endowed with an almost product structure $P$ compatible with the induced metric and therefore with its conformal class, see Section 3.

Now, our main results are Theorem \ref{251121A} and Corollary \ref{261121A} which characterize when the standard  tractor conformal bundle corresponding to a spacelike submanifold $\Psi\colon M^{n}\to (\widetilde{M}^{n+2}, \widetilde{g})$ and $\xi$ as above is normal. The normality condition on $(\mathcal{T},\mathcal{T}^{1}, \widetilde{g} , \widetilde{\nabla})$ is stated in terms of relationships between the extrinsic and intrinsic geometry of the spacelike submanifold. Theorem \ref{251121A} deals with the general case $A^{2}_{\xi}= \lambda^{2} \cdot \mathrm{Id}$ and  Corollary \ref{261121A} with the umbilical one. Although, the normality condition for a tractor connection was stated for $\mathrm{dim}M=n\geq 3$, the curvature properties in Theorem \ref{261221A} have sense for $n\geq 2$. 
The main results of this paper can be summarized as follows.

\begin{quote}
Let $\Psi:M^n\rightarrow \widetilde{M}^{n+2}$ be a spacelike submanifold in a Lorentzian manifold with induced metric $g$ and $\xi\in \mathfrak{X}^{\perp}(M^n)$ a lightlike vector field. Let us consider $(\mathcal{T},\mathcal{T}^{1}, \widetilde{g} , \widetilde{\nabla})$. Then,
\begin{enumerate}
    \item The induced connection $\widetilde{\nabla}$ is tractor connection if and only if the Weingarten endomorphism $A_{\xi}$ is not singular at every point.
    
    \item $(\mathcal{T},\mathcal{T}^{1}, \widetilde{g} , \widetilde{\nabla})$ is standard for the induced metric $g$ if and only if there is a smooth non-vanishing function $\lambda\in C^{\infty}(M^n)$ such that $A_{\xi}=\lambda^{2}\cdot \mathrm{Id}.$

    \item Assume $A_{\xi}=\lambda \cdot \mathrm{Id}$ and there is a lightlike vector field $\eta \in \mathfrak{X}^{\perp}(M^n)$ such that $\widetilde{g}(\xi, \eta)=-1$. Then $(\mathcal{T}, \mathcal{T}^{1}, \widetilde{g}, \widetilde{\nabla})$ is normal if and only if the following conditions hold
    
    \begin{enumerate}
        \item $\nabla^{\perp}\xi=\frac{1}{\lambda}d\lambda \otimes \xi$, where $\nabla^{\perp}$ denotes the normal connection.

        \item For every $X,Y\in \mathfrak{X}(M)$, the  Ricci tensor of $g$ satisfies
$$
\mathrm{Ric}(X, Y)=\frac{n}{2}\| \mathbf{H}\|^{2}\,g(X,Y)- (n-2) \widetilde{g}(\mathbf{H}, \xi) g(X, A_{\eta}Y),
$$
where $\mathbf{H}$ is the mean curvature vector field of $\Psi\colon M^n \to \widetilde{M}^{n+2}.$
    \end{enumerate}
\end{enumerate}

\end{quote}

As a direct consequence when $(\mathcal{T}, \mathcal{T}^{1}, \widetilde{g}, \widetilde{\nabla})$ is normal, the scalar curvature $S$ of the induced metric $g$ satisfies $S=n(n-1)\| \mathbf{H}\|^{2}$ and $(M,g)$ is Einstein if and only if $\Psi:M^n\rightarrow \widetilde{M}^{n+2}$ is totally umbilical.

The paper ends with an application to Lorentzian warped  products 
$$
\left(\widetilde{M}^{n+2}=B\times M^{n}, \widetilde{g}=\pi_{B}^{*}(g_{B}) + \varphi(\rho, s)^{2}\,\pi_{M}^{*}(g)\right),
$$
where $(M^{n},g)$ is a Riemannian manifold and  $B \subset \R^{2}$  an open subset with natural coordinates $(\rho, s)$ endowed with a Lorentzian metric $g_{B}$ such that there are two lightlike  vector fields fields  $\xi, \eta \in \mathfrak{X}(B)$ with $g_{B}(\xi, \eta)=-1$. If we assume there is $(\rho_{0}, s_{0})\in B$ such that $ \varphi(\rho_{0}, s_{0})=1$, the spacelike immersion $\Psi \colon M^{n} \to \widetilde{M}^{n+2}$ given by $\Psi(x)=(\rho_{0},s_{0}, x)$ satisfies $\Psi^{*}(\widetilde{g})=g$.
Assume  $\xi_{(\rho_{0}, s_{0})} \varphi \neq 0$. Then, the corresponding tractor conformal bundle is always standard for the conformal class of the induced metric and it is normal if and only if $(M^{n},g)$ is an Einstein manifold with 
$\mathrm{Ric}=\Lambda\, g$.  In such a case, we have  
$
\Lambda=(n-1)\|\mathbf{H}\|^{2},
$
Proposition \ref{131221A}.

\section{Preliminaries}\label{Prel}

  In this section we fix some terminology and notations for spacelike immersions in Lorentzian manifolds, that is, immersions of a manifold $M$ in a Lorentzain manifold $(\widetilde{M}, \widetilde{g})$ such that the induced metric on $M$ is Riemannian.
This Section ends with the result which states the normality condition for a standard tractor conformal bundle for a conformal structure $(M, c)$, \cite{Cap}.

Let $(\widetilde{M}, \widetilde{g})$ be an $m$-dimensional Lorentzian manifold. A smooth immersion  $\Psi:M^n\rightarrow \widetilde{M}^m$ of a (connected) $(n\geq 2)$-dimensional  manifold $M$  is said to be a spacelike when the induced metric $g:= \Psi^{*}(\widetilde{g})$ is Riemannian. 
Let $\widetilde{\nabla}$ and $\nabla$ be the Levi-Civita
connections of  $\widetilde{M}$ and $M$, respectively. Let
$\nabla^{\perp}$ be the connection on the normal bundle. The Gauss
and Weingarten formulas are
$$\widetilde{\nabla}_X Y=\nabla_XY + \mathrm{II}(X,Y)
\, \quad \mathrm{and} \, \quad
\widetilde{\nabla}_X\xi=-A_{\xi}X+\nabla^{\perp}_X\,\xi,$$
for any $X,Y \in \mathfrak{X}(M^{n})$ and $\xi \in
\mathfrak{X}^{\perp}(M^{n})$, where $\mathrm{II}$ denotes the
second fundamental form and $A_{\xi}$ the shape (or Weingarten) operator corresponding to $\xi$. As usual, we have agreed to ignore the differential of the map $\Psi$ when no confusion can arise.
The shape operator $A_{\xi}$
corresponding to $\xi$ is related to $\mathrm{II}$ by
$$g( A_{\xi}X, Y ) = g( \mathrm{II}(X,Y), \xi
),$$
for all $X,Y \in \mathfrak{X}(M^{n})$. In particular, $A_{\xi}$ is $g$-self adjoint.
The mean curvature vector field of $\Psi$ is given by
$\mathbf{H}=\frac{1}{n}\,\mathrm{trace}_{_{g}}\mathrm{II}.$

\smallskip

From now on we make the assumptions $m=n+2$ and there is a lightlike normal vector field $\xi \in
\mathfrak{X}^{\perp}(M^{n})$ along $\Psi$. 
Under these assumptions, there is a natural choice of a tractor conformal bundle on $M$ as follows.
The vector bundle is $\mathcal{T}\to M$ where $\mathcal{T}=\Psi^{*}(T\widetilde{M})$ is the
 pull-back via $\Psi $ of the tangent bundle $T\widetilde{M}\to \widetilde{M}$. The Lorentzian bundle metric $h$
is the metric $\widetilde{g}$ and the distinguished lightlike line subbundle $\mathcal{T}^{1}=\mathrm{Span}(\xi).$ 
Let $\overline{\mathfrak{X}}(M)$ be the $C^{\infty}(M)-$module of vector fields along the immersion $\Psi$. 
If we also denote by $\widetilde{\nabla}\colon \mathfrak{X}(M)\times \overline{\mathfrak{X}}(M)\to \overline{\mathfrak{X}}(M)$  the induced connection of $\Psi:M^n\rightarrow \widetilde{M}^m$,
it is well-known that $\widetilde{\nabla}\widetilde{g}=0.$ 
Taking into account that $\Psi:M^n\rightarrow \widetilde{M}^{n+2}$ is a codimension two spacelike immersion, there is $\omega \in \Omega^{1}(M, \R)$ such that  $\nabla^{\perp}\xi= \omega \otimes \xi$. 

If we assume there is another  lightlike normal vector field $\eta \in \mathfrak{X}^{\perp}(M^{n})$
such that  $\widetilde{g}(\xi,\eta)=-1$.  Then, for every $X,Y\in\mathfrak{X}(M)$, the second fundamental form can be written as
\begin{equation}\label{270221B}
\mathrm{II}(X,Y)=-g(A_{\eta}X,Y)\xi-g(A_{\xi}X,Y)\eta,
\end{equation}
and  the mean curvature vector field satisfies
\begin{equation}\label{H}
\mathbf{H}=-\dfrac{1}{n}\left(\mathrm{trace}\,(A_{\eta})\xi+\mathrm{trace}\,(A_{\xi})\eta\right).
\end{equation}

Our assumption on the curvature sign is $R(X,Y)Z=\nabla_{X}\nabla_{Y}Z-\nabla_{Y}\nabla_{X}Z-\nabla_{[X,Y]}Z$.  Let us recall the Gauss equation  (e.g., \cite[Chap. 3]{One83}).  Observe that our assumption on the sign of $R$ is the opposite to \cite{One83}.
\begin{equation}\label{241121C}
g(R(X,Y)Z, V)=\widetilde{g}(\widetilde{R}(X,Y)Z, V)- \widetilde{g}(\mathrm{II}(X,Z), \mathrm{II}(Y,V))+ \widetilde{g}(\mathrm{II}(X,V), \mathrm{II}(Y,Z)),
\end{equation}
for any $X,Y,Z, V\in \mathfrak{X}(M).$

\bigskip
For the sake of completeness, let us recall the following result \cite{Cap}.

\begin{theorem}\label{261221A}$(\cite{Cap})$
Let $(\mathcal{T}, \mathcal{T}^{1}, h)$ be a standard tractor conformal bundle for a conformal structure $(M, c)$ with $\mathrm{dim }\,M=n\geq 3$. A tractor connection $\nabla^{\mathcal{T}}$ is normal if and only if  the following conditions hold, 
\begin{enumerate}
\item Its curvature $R^{\mathcal{T}}$ satisfies
$
R^{\mathcal{T}}(X, Y)\sigma\subset \Gamma(\mathcal{T}^{1}),
$
for every $X, Y \in \mathfrak{X}(M)$ and $\sigma \in \Gamma(\mathcal{T}^{1})$. Thus, $R^{\mathcal{T}}(X, Y)$ induces an endomorphismon  $(\mathcal{T}^{1})^{\perp}/\mathcal{T}^{1}$. In fact, by means of any section $\sigma\in \Gamma(\mathcal{T}^{1})$, we can consider $W\in \Gamma(\Lambda^{2}T^{*}M \otimes L(TM, TM))$ as follows
$$
W(X, Y )Z=(\beta_{\sigma})^{-1}\Big( R^{\mathcal{T}}(X, Y)\beta_{\sigma}(Z)\Big),\quad  X,Y, Z \in \mathfrak{X}(M).
$$
Note that $W$ does not depend on the choice of $\sigma \in \Gamma(\mathcal{T}^{1}).$
\item The Ricci type contraction of  $W\in \Gamma(\Lambda^{2}T^{*}M \otimes L(TM, TM))$
vanishes on $M$. This condition is equivalently stated as follows
$$
\sum_{i=1}^{n}h^{\sigma}\Big(W(E_{i}, X )Y, E_{i} \Big)=\sum_{i=1}^{n}h\Big( R^{\mathcal{T}}(E_{i},X)\beta_{\sigma}(Y), \beta_{\sigma}(E_{i})\Big)=0,
$$
for any $X,Y \in \mathfrak{X}(M)$,
where $(E_{1}, \dots , E_{n})$ is a local orthonormal  frame  with respect to the Riemannian metric $h^{\sigma}\in c$ as in $(\ref{241121B})$ for any $\sigma \in \Gamma(\mathcal{T}^{1}).$

\end{enumerate}

\end{theorem}

\section{Main results}

\begin{proposition}\label{241121D}
Let  $\Psi:M^n\rightarrow (\widetilde{M}^{n+2}, \widetilde{g})$ be a spacelike immersion and consider  $(\mathcal{T},\mathcal{T}^{1}, \widetilde{g} , \widetilde{\nabla})$ as above. Then,
$\widetilde{\nabla}$ is a tractor connection if and only if the Weingarten endomorphism $A_{\xi}$ is non-singular at every point $x\in M.$
In this case, the induced metric from $\Psi$ is in the equivalence class of the conformal structure deduced from the tractor connection $\widetilde{\nabla}$ if and only if there is a smooth function $\lambda\in C^{\infty}(M)$ such that $A^{2}_{\xi}= \lambda^{2} \cdot \mathrm{Id}$ (where $\lambda$ is a non-vanishing function on $M$).
\end{proposition}
\begin{proof} As was mentioned, the condition $\widetilde{\nabla}h=0$ always holds.
In our case, by means of the Weingarten formula, the vector bundle homomorphism
$\beta$ given by (\ref{011021}) is determined by
\begin{equation}\label{180921}
\beta(X)(\xi)=\widetilde{\nabla}_{X}\xi +\,\mathcal{T}^{1}=- A_{\xi}(X)+\omega(X)\xi+\,\mathcal{T}^{1}=- A_{\xi}(X)+\,\mathcal{T}^{1},
\end{equation}
where $X\in \mathfrak{X}(M)$. That is, $\beta_{\xi}(X)=- A_{\xi}(X)+\,\mathcal{T}^{1}$. From (\ref{180921}), it is clear that $\beta$ is an isomorphism of vector bundles if and only if $A_{\xi}$ is non-singular at all points of $M$. Moreover, from (\ref{180921}), every section $\sigma=f \, \xi\in\Gamma(\mathcal{T}^1)$ produces the Riemannian metric 
 \begin{equation}\label{011021B}
 f^{2}\,g(A_{\xi}(X),A_{\xi}(Y))= f^{2}\, g(A^{2}_{\xi}(X), Y),
 \end{equation}where $X,Y\in \mathfrak{X}(M)$. Then, the metric given in (\ref{011021B}) belongs to the conformal class of the induced metric $g$ if and only if there is $\lambda\in C^{\infty}(M)$ with $A^{2}_{\xi}=\lambda^{2}\cdot \mathrm{Id}$.
\end{proof}

\begin{definition}
Let  $\Psi:M^n\rightarrow (\widetilde{M}^{n+2}, \widetilde{g})$ be a spacelike immersion and $\xi\in \mathfrak{X}^{\perp}(M^n)$ a fixed lightlike normal vector field.
Assume there is a non-vanishing smooth function $\lambda\in C^{\infty}(M)$ such that $A^{2}_{\xi}= \lambda^{2} \cdot \mathrm{Id}$. The standard tractor conformal bundle with tractor connection $(\mathcal{T}, \mathcal{T}^{1}, \widetilde{g}, \widetilde{\nabla})$ on $M$, where $\mathcal{T}=\Psi^{*}(T\widetilde{M})$ and $\mathcal{T}^{1}=\mathrm{Span}(\xi)$, is said to be the associated to the pair $(\Psi:M^n\rightarrow \widetilde{M}^{n+2}, \xi)$.
\end{definition}

\begin{remark}\label{241121E}
{\rm
For associated tractors to codimension two spacelike immersions as above,  the induced metric satisfies $g=\widetilde{g}^{\frac{1}{\lambda}\xi},
$
according to the notation in (\ref{241121B}).
}
\end{remark}

Under the assumptions of Proposition \ref{241121D} and taking into account that $M$ is assumed to be connected, there are two mutually disjoint possibilities. Namely, $\xi$ is an umbilical direction, that is $A_{\xi}=\lambda\cdot \mathrm{Id}$ on $M$ or the tensor field $P:=\frac{1}{\lambda}\cdot A_{\xi}$ defines an almost product structure (with $P\neq \pm \mathrm{Id}$) on the Riemannian manifold $(M^{n}, g)$. That is, we have $P^{2}= \mathrm{Id}$ and $g(PX, PY)=g(X,Y)$ for any $X,Y \in \mathfrak{X}(M).$

Let us denote by $\mathrm{Ric}$ and $S$ the Ricci tensor and the scalar curvature on $(M^{n},g)$, respectively. For a fixed almost product structure $(M^{n},P,g)$ on a Riemannian manifold, the $P$-Ricci tensor $\mathrm{Ric}^{*}$ and the $P$-scalar curvature $S^{*}$ are defined by (see \cite{Mekerov2008})
$$
\mathrm{Ric}^{*}(X,Y)=\sum_{i=1}^{n}g\left(R(E_{i},X)Y, PE_{i}\right), \quad S^{*}=\sum_{i=1}^{n}\mathrm{Ric}^{*}(E_{i}, E_{i}),\quad X,Y\in \mathfrak{X}(M),
$$
where $(E_{1}, \dots , E_{n})$ is a local orthonormal frame.

In order to shorten the statement of the following result,
we will write $P=\frac{1}{\lambda}A_{\xi}$ in both cases, namely, when $\xi$ is an umbilical direction (i.e., $P=\mathrm{Id}$) and for the (proper) almost product structure. Of course, when $\xi$ is an umbilical direction, we have $\mathrm{Ric}^{*}=\mathrm{Ric}$ and $S^{*}=S$.

\begin{theorem}\label{251121A}
Let  $(\mathcal{T}, \mathcal{T}^{1}, \widetilde{g}, \widetilde{\nabla})$ be the associated tractor to the pair $(\Psi:M^n\rightarrow \widetilde{M}^{n+2}, \xi)$ with $A_{\xi}=\lambda^{2}\cdot \mathrm{Id}$, where $\lambda$ is a non-vanishing function on $M$. Assume there is a lightlike normal vector field $\eta \in \mathfrak{X}^{\perp}(M^n)$ such that $\widetilde{g}(\xi, \eta)=-1$.
Then, the tractor connection $\widetilde{\nabla}$
is normal if and only if the following conditions hold
\begin{enumerate}
 
\item For every $X,Y\in \mathfrak{X}(M)$, we have
\begin{equation}\label{cond1}
(\nabla_{X}A_{\xi})(Y)-(\nabla_{Y}A_{\xi})(X)=\omega(X)A_{\xi} Y-\omega(Y)A_{\xi} X.
\end{equation}

%
%
\item For every $X,Y\in \mathfrak{X}(M)$, the $P$-Ricci tensor satisfies
\begin{equation}\label{cond2}
\mathrm{Ric}^{*}(X, Y)=-\mathrm{trace}(A_{\xi}\circ A_{\eta})\,g(X,PY)- (n-2) \lambda \,g(X, A_{\eta}Y).
\end{equation}

\end{enumerate}
\end{theorem}
\begin{proof} A direct computation shows for the curvature tensor of $\widetilde{\nabla}$,
\begin{equation*}
\begin{split}
\widetilde{R}(X,Y)\xi & = \widetilde{\nabla}_{X}\widetilde{\nabla}_{Y}\xi- \widetilde{\nabla}_{Y}\widetilde{\nabla}_{X}\xi-  \widetilde{\nabla}_{[X,Y]}\xi \\ & = d\omega(X,Y)\xi- \nabla_{X}(A_{\xi}Y)- \mathrm{II}(X, A_{\xi}Y)
- \omega(Y)A_{\xi}X \\
 & + \nabla_{Y}(A_{\xi}X)+\mathrm{II}(A_{\xi}X, Y)+ \omega(X) A_{\xi}Y+ A_{\xi}([X,Y]),
\end{split}
\end{equation*}
where $X,Y\in \mathfrak{X}(M).$ A direct computation shows
$$
h(-\mathrm{II}(X, A_{\xi}Y)+\mathrm{II}( A_{\xi}X,Y) , \xi)=
-g( A_{\xi}X, A_{\xi}Y)+g( A_{\xi}X, A_{\xi}Y)=0.
$$
Now, taking into account that $M$ is a codimension two spacelike immersion, the term $-\mathrm{II}(X, A_{\xi}Y)+\mathrm{II}(Y, A_{\xi}X)$ must be collinear with $\xi$. Therefore, the first normality condition for $\widetilde{\nabla}$ reduces to
$$
-\nabla_{X}(A_{\xi}Y)
- \omega(Y)A_{\xi}X+
\nabla_{Y}(A_{\xi}X)+ \omega(X) A_{\xi}Y+ A_{\xi}([X,Y])=0
$$
which is equivalently written as (\ref{cond1}).

Assume now  equation (\ref{cond1}) holds.  From (\ref{180921})   and Remark \ref{241121E},  the second normality condition for $\widetilde{\nabla}$ is equivalent to
$$
\sum_{i=1}^{n}g\Big(W(E_{i}, X )Y, E_{i} \Big)=\frac{1}{\lambda^{2}}\sum_{i=1}^{n}\widetilde{g}\Big( \widetilde{R}(E_{i},X)A_{\xi}Y, A_{\xi}E_{i}\Big)=0,
$$
where $(E_{1}, \dots, E_{n})$ is a local orthonormal frame with respect to $g$.
By means of the Gauss equation (\ref{241121C}), we have

\begin{equation*}
\begin{split}
\sum_{i=1}^{n}\widetilde{g}\Big( \widetilde{R}(E_{i},X)A_{\xi}Y, A_{\xi}E_{i}\Big) & = \sum_{i=1}^{n}g\Big( R(E_{i}, X)A_{\xi}Y, A_{\xi}E_{i}\Big)+ \widetilde{g}\Big(\mathrm{II}(E_{i},A_{\xi}Y), \mathrm{II}(X,A_{\xi}E_{i})\Big) \\ & -\widetilde{g}\Big( \mathrm{II}(E_{i},A_{\xi}E_{i}), \mathrm{II}(X,A_{\xi}Y)\Big).
\end{split}
\end{equation*}

Now,  taking into account (\ref{270221B}),  straightforward computations show that
$$
\sum_{i=1}^{n} \widetilde{g}\Big(\mathrm{II}(E_{i},A_{\xi}Y), \mathrm{II}(X,A_{\xi}E_{i})\Big)=-2 \lambda^{2}g(X,  A_{\eta}(A_{\xi}Y))
$$
and
$$
\sum_{i=1}^{n}\widetilde{g}\Big( \mathrm{II}(E_{i},A_{\xi}E_{i}), \mathrm{II}(X,A_{\xi}Y)\Big)=-\lambda^{2}\Big(\mathrm{trace}(A_{\xi}\circ A_{\eta})g(X,Y)+ n\, g(X, A_{\eta}(A_{\xi}Y))\Big).
$$
On the other hand,  for the curvature term we get
$$
\sum_{i=1}^{n}g\Big( R(E_{i}, X)A_{\xi}Y, A_{\xi}E_{i}\Big)=\lambda^{2}\sum_{i=1}^{n}g\Big( R(E_{i}, X)PY, PE_{i}\Big)=\lambda^{2}\mathrm{Ric}^{*}(X, PY).
$$
Therefore,  the second normality condition for $\widetilde{\nabla}$ writes as follows
$$
\mathrm{Ric}^{*}(X, PY)+\mathrm{trace}(A_{\xi}\circ A_{\eta})g(X,Y)+ (n-2) g(X, A_{\eta}(A_{\xi}Y))=0.
$$
Since $P^{2}=\mathrm{Id}$,  one easily shows that this formula is equivalent to (\ref{cond2}).
\end{proof}

\begin{remark}
{\rm For a Lorentzian manifold $(\widetilde{M}^{n+2}, g)$ of constant sectional curvature, the equation (\ref{cond1}) is a direct consequence of the Codazzi equation. Therefore, for a spacelike immersion $\Psi:M^n\rightarrow \widetilde{M}^{n+2}$, where $\widetilde{M}^{n+2}$ has constant sectional curvature, the associated tractor has normal tractor connection if and only if (\ref{cond2}) holds.}
\end{remark}

In the particular case that $\xi$ is an umbilical direction with $A_{\xi}=\lambda\cdot \mathrm{Id}$, we have $\lambda=\widetilde{g}(\mathbf{H}, \xi)$ and then, Theorem \ref{251121A} reads as follows.

\begin{corollary}\label{261121A}
Let  $(\mathcal{T}, \mathcal{T}^{1}, \widetilde{g}, \widetilde{\nabla})$ be the associated tractor to the pair $(\Psi:M^n\rightarrow \widetilde{M}^{n+2}, \xi)$ with $\xi$ an umbilical direction with $A_{\xi}=\lambda\cdot \mathrm{Id}$ and $\lambda> 0$. Assume there is a lightlike normal vector field $\eta \in \mathfrak{X}^{\perp}(M^n)$ such that $\widetilde{g}(\xi, \eta)=-1$.
Then, the tractor connection $\widetilde{\nabla}$
is normal if and only if 
$\omega=\frac{1}{\lambda}d\lambda$ and 
for every $X,Y\in \mathfrak{X}(M)$, the  Ricci tensor of $g$ satisfies
\begin{equation}\label{cond4}
\mathrm{Ric}(X, Y)=\frac{n}{2}\| \mathbf{H}\|^{2}\,g(X,Y)- (n-2) \widetilde{g}(\mathbf{H}, \xi) g(X, A_{\eta}Y).
\end{equation}

\end{corollary}

\begin{remark}
{\rm Under the assumptions of Corollary \ref{261121A}, the scalar curvature of $g$ is $S=n(n-1)\| \mathbf{H}\|^{2}$. This formula widely generalizes to \cite[Cor. 4.5]{PPR}. For  $\mathrm{dim}\, M=n\geq 3$, the Schouten tensor
is (compare with \cite[Theor. 5.6]{MP})
$$
P^{g}(X,Y)=- \widetilde{g}(\mathbf{H}, \xi)\,g(X, A_{\eta}Y).
$$
The spacelike submanifold $(M,g)$ is Einstein if and only if $\Psi:M^n\rightarrow \widetilde{M}^{n+2}$ is totally umbilical. In such a case, we have $\mathrm{Ric}=(n-1)\| \mathbf{H}\|^{2}\,g.$ For $n=2$, equation (\ref{cond4}) reduces to $\mathrm{Ric}(X, Y)=\| \mathbf{H}\|^{2}\,g(X,Y)$. Therefore, equation (\ref{cond4}) holds if and only if $|\mathbf{H}\|^{2}=K^g$, where $K^g$ is the Gauss curvature of the induced metric $g$.}
\end{remark}

\begin{remark}
{\rm Let us consider $\widetilde{M}^{n+2}=\mathbb{L}^{n+2}$, where $\mathbb{L}^{n+2}$ denotes the $(n+2)$-dimensional Lorentz-Minkowski spacetime. Then, the condition $d\omega=\frac{1}{\lambda}d\lambda$ with $\lambda >0$ is equivalent to $\Psi:M^n\rightarrow \widetilde{M}^{n+2}$ factors, up to a translation, through the lightlike cone $\mathcal{N}^{n+1}=\{v \in \mathbb{L}^{n+1}: \langle v, v\rangle=0, v\neq 0\}$ \cite{PR13}. A classical result due to Brinkmann \cite{Br} states that a simply connected Riemannian manifold $(M^n, g)$ with $n\geq 3$ is (locally) conformally flat if and only if can be isometrically immersed into $\mathbb{L}^{n+2}$ through the lightlike cone $\mathcal{N}^{n+1}$ (see \cite{AD} for a proof in a modern form). It is a direct consequence of the Uniformation Theorem that every $2$-dimensional simply connected Riemannian manifold $(M^2, g)$ can be isometrically immersed into $\mathbb{L}^{4}$ through the lightlike cone $\mathcal{N}^{3}$.
For spacelike submanifolds $\Psi\colon M^n \to \mathbb{L}^{n+2}$ through the lightlike cone $\mathcal{N}^{n+1}$, the position vector field $\Psi \in \mathfrak{X}^{\perp}(M^n)$ is lightlike with $A_{\Psi}=-\mathrm{Id}$. Therefore, for $\xi=- \Psi$ and from the Gauss equation, we have for the Ricci tensor of the induced metric $g$ (see \cite{PPR} and take into account that the normalization condition in \cite{PPR}  is $\widetilde{g}(\xi, \eta)=1$),
\begin{equation}\label{021221A}
\mathrm{Ric}(X,Y)=ng(A_{\mathbf{H}}X, Y)+2g(X, A_{\eta}Y). 
\end{equation}
Now, since $\lambda=1$, a direct computation shows that equations (\ref{cond4}) and (\ref{021221A}) are the same. Therefore, as it is well-known for conformally flat Riemannian manifolds, we get a tractor normal connection by means of an isometric immersion through the lightlike cone of Lorentz-Minkowski spacetime.
}
\end{remark}

\section{Example}

Let $(M^{n},g)$ be a Riemannian manifold and  $B \subset \R^{2}$  an open subset with natural coordinates $(\rho, s)$.  Let  $\varphi \in C^{\infty}(B)$ be a smooth function with $\varphi >0$.  We write $\pi_{B}$ and $\pi_{M}$ for the projections from $\widetilde{M}^{n+2}:=B\times M^{n}$  onto $B$ and $M$, respectively.
Let us consider  the family of Lorentzian warped metrics on $\widetilde{M}$, \cite[Chap. 7]{One83}
\begin{equation}\label{291121A}
\widetilde{g}=\pi_{B}^{*}(g_{B}) + \varphi(\rho, s)^{2}\,\pi_{M}^{*}(g),
\end{equation}
where $g_{B}$ is any Lorentzian metric on $B$ such that there are two lightlike  vector fields fields  $\xi, \eta \in \mathfrak{X}(B)$ with $g_{B}(\xi, \eta)=-1$. We use the same notation for  vector fields on $B$ or $M^{n}$ and for theirs lifts to the product manifold $\widetilde{M}$.

Assume there is $(\rho_{0}, s_{0})\in B$ such that $ \varphi(\rho_{0}, s_{0})=1$.  In such a case, 
the spacelike immersion $\Psi \colon M^{n} \to \widetilde{M}^{n+2}$ given by $\Psi(x)=(\rho_{0},s_{0}, x)$ satisfies $\Psi^{*}(\widetilde{g})=g$ and the vector fields $\xi$  and $\eta$ are
lightlike normal vector fields along $\Psi$ with $\widetilde{g}(\xi, \eta)=-1$.
Since $\widetilde{g}$ is a warped metric,  one can apply \cite[Prop. 7.35]{One83} to obtain for every $X \in \mathfrak{X}(M)$,
\begin{equation}\label{02122021}
\widetilde{\nabla}_{X}\xi=(\xi_{(\rho_{0}, s_{0})} \varphi )\cdot X, \quad \mathrm{ and } \quad \widetilde{\nabla}_{X}\eta=(\eta_{(\rho_{0}, s_{0})} \varphi )\cdot X.
\end{equation}
Moreover, the mean curvature vector field of $\Psi$ is given by
$ \mathbf{H}=- \nabla \varphi (\rho_{0}, s_{0})
$, where $\nabla \varphi$ is the gradient of $\varphi$ in $B$.
Under this notation, we have.

\begin{proposition}\label{131221A}
Let  $(\mathcal{T}, \mathcal{T}^{1}, \widetilde{g}, \widetilde{\nabla})$ be the associated tractor to the pair 
$$(\Psi:M^n\rightarrow \widetilde{M}^{n+2}, \xi)$$ as above. 
Assume  $\xi_{(\rho_{0}, s_{0})} \varphi \neq 0$. Then, $(\mathcal{T}, \mathcal{T}^{1}, \widetilde{g}, \widetilde{\nabla})$ is always standard for $g$ and the tractor connection $\widetilde{\nabla}$
is normal if and only if $(M^{n},g)$ is an Einstein manifold with 
$\mathrm{Ric}=\Lambda\, g$.  In such a case, we have  
$
\Lambda=(n-1)\|\mathbf{H}\|^{2}=(n-1)\|\nabla\varphi (\rho_{0}, s_{0})\|^{2}.
$
\end{proposition}
\begin{proof}
From (\ref{02122021}) we get $\omega=0$, so the condition $\omega=\frac{1}{\lambda}d\lambda$ is direct since $\lambda=-\xi_{(\rho_{0}, s_{0})}\varphi$ is constant along $M$. The condition on the Ricci tensor is obtained from the equations (\ref{H}), (\ref{cond4}) and (\ref{02122021}), from which it follows that
$$
\mathrm{Ric}(X,Y)=-2(n-1)\,\xi_{(\rho_{0}, s_{0})}\varphi\,\eta_{(\rho_{0}, s_{0})}\varphi \,g(X,Y),
$$
for $X,Y\in\mathfrak{X}(M)$. Since $\|\nabla\varphi (\rho_{0}, s_{0})\|^{2}=-2\,\xi_{(\rho_{0}, s_{0})}\varphi\,\eta_{(\rho_{0}, s_{0})}\varphi$ we get
the announced result.

\end{proof}

\begin{remark}\label{141221B}
{\rm

Let us recall the notion of ambient manifolds for Riemannian conformal structures as appeared in \cite{Cap}. Let $(M, [g])$ be a conformal Riemannian structure on a $(n\ge2)$-dimensional manifold $M$.
Let ${\bf N}\subset \mathrm{Sym}^{+}(M)$ be the fiber bundle of scales of the conformal structure and $\pi :{\bf N}\to M^{n}$ the projection. Every section of $\pi$ provides a  representative $g\in [g]$.
On ${\bf N}$ there are a free $\R_{>0}$-action (called dilations) defined by $\delta(s)(u)= s^{2}u$ for $u \in {\bf N}$ and a tautological symmetric $2$-tensor $h$ given by $h(\xi, \eta)=u(T_{u}\pi \cdot \xi, T_{u}\pi \cdot \eta)$ for  $\xi, \eta \in T_{u}{\bf N}$. An ambient manifold is a $(n+2)$-dimensional manifold ${\bf M}$ endowed with
a free $\R_{>0}$-action also denotes by $\delta$, and a $\R_{>0}$-equivariant embedding $i:{\bf N} \to {\bf M}$, that is, $i(\delta(s)(u))=\delta(s)(i(u))$. 
Then, an ambient metric is a Lorentzian metric $g_{_{L}}$ on ${\bf M}$ such that the following conditions hold.
\begin{enumerate}
\item The metric is homogeneous of degree $2$ with respect to the $\R_{>0}$-action. That is, if $Z\in \mathfrak{X}({\bf M})$ denotes the fundamental vector field generating the $\R_{>0}$-action, then we have $\mathcal{L}_{Z}g_{_{L}}=2g_{_{L}}$.
\item For all $u=g_{x}\in {\bf N}$ and $\xi, \eta \in T_{u}{\bf N}$, we have $i^{*}(g_{_{L}})(\xi, \eta)=g_{x}(T_{u}\pi \cdot\xi, T_{u}\pi \cdot\eta)$. That is, $i^{*}(g_{_{L}})=h$.
\end{enumerate}
The original definition by Fefferman-Graham of an ambient metric in \cite{F-G} adds certain conditions of Ricci-flatness for the metric $g_{_{L}}$.
In \cite{Cap}, the authors have related the ambient metric construction with conformal normal tractor bundles as follows.
\begin{quote}
Let $g_{_{L}}$ be an ambient metric such that $d\alpha|_{\bf N}=0$ where $\alpha$ is the one-form metrically equivalent to the infinitesimal generator $Z\in \mathfrak{X}(\bf{M})$  of the action $\delta$. Then, the standard tractor bundle $(\mathcal{T}, g_{_{L}}, \nabla)$ is normal if and only if $\mathrm{Ric}^{g_{_{L}}}(U,V)|_{\bf{N}}=0$ for every  $U,V\in \mathfrak{X}({\bf N})$. 
\end{quote}
When $n\geq 3$ and the conformal structure $[g]$ contains an Einstein metric $g$ with $\mathrm{Ric}=\Lambda \, g$, then a Ricci flat ambient metric can be explicitly given by ${\bf M}=B\times M$ where $B=(-\varepsilon , +\varepsilon) \times \R_{>0}$ endowed  with the Lorentzian metric
 $g_{_{B}}=ds \otimes d(\rho s)+ d(\rho s) \otimes ds$ for the natural coordinates $(s,\rho)$
 and $ g_{_{\bf{M}}}$
is the Lorentzian warped metric
$$
 g_{_{\bf{M}}}=g_{_{B}}+s^{2}\left(1+ \frac{\Lambda}{2(n-1)} \rho\right)^{2}g.
$$
Proposition \ref{131221A} gives a partial converse to this fact.
}
\end{remark}

{\small DPTO. DE MATEMÁTICA APLICADA, UNIVERSIDAD DE MÁLAGA, 29071-MÁLAGA (SPAIN)}

{\it E-mail address:} \textcolor[gray]{0.5}{ruyman@uma.es}

{\it E-mail address:} \textcolor[gray]{0.5}{fpalomo@uma.es}
\end{document}